\documentclass[12pt]{article}
\usepackage[T1]{fontenc}
\usepackage{amsmath,amssymb,amsthm}
\usepackage{booktabs}
\usepackage{url}
\usepackage[hidelinks]{hyperref}

\newtheorem{theorem}{Theorem}[section]
\newtheorem{lemma}[theorem]{Lemma}

\newtheorem*{continuationlemma}{Lemma}
\theoremstyle{definition}

\theoremstyle{remark}

\numberwithin{equation}{section}

\newcommand{\supp}{\operatorname{supp}}
\newcommand{\Full}{\mathcal M}
\newcommand{\Shadow}{\mathcal R}
\newcommand{\EW}{\mathrm{EW}}

\title{Surjectivity of Finite Rank-Capped Enots--Wolley Sequences}
\author{Nathan Myles Nichols}
\date{September 2026}

\begin{document}
\maketitle

\begin{abstract}
For each fixed integer $K\ge2$, we consider the Enots--Wolley sequence in
which every term after the initial $1,2$ is required to have between two and
$K$ distinct prime divisors.  We prove that every integer satisfying this
restriction occurs.  If an exact prime support $T$ were selected only
finitely often, then after a finite cutoff the terms meeting $T$ would form
short episodes, and every full term would force an earlier proper term at
comparable numerical height.  The case $|T|=K$ is then ruled out directly.  In
the remaining case $|T|<K$, at a record proper value $H$, greediness forces
every unblocked rank-$K$ integer below $H$ containing exactly one prime of $T$
to have occurred earlier.  Fixed-order Landau estimates show that this proper
population has order $H(\log\log H)^{K-2}/\log H$, while the entire possible
full population on the same scale has strictly smaller logarithmic order.  This
contradiction proves surjectivity.  The theorem concerns each fixed rank cap
and does not settle surjectivity of the unrestricted Enots--Wolley sequence.
\end{abstract}

\section{Introduction}
\label{sec:introduction}

A sequence is \emph{lexicographically earlier} than another if its term at the
first index where they differ is smaller.  A lexicographically earliest
sequence subject to given conditions therefore makes the smallest possible
choice at each index that is compatible with a complete sequence satisfying
those conditions.  When the sequence is required to be infinite, this is not
necessarily the same as taking the smallest unused integer that passes the
conditions visible at the current step: such a choice might prevent the
sequence from continuing.

Lexicographically earliest divisibility sequences provide concrete examples of
this distinction and of the difficulty of determining the range.  The
Yellowstone sequence, OEIS A098550, begins with $1,2,3$ and subsequently
chooses the least unused integer sharing a prime divisor with the term two
places earlier and coprime to its immediate predecessor.  Applegate et al.
proved that it is a permutation of the positive integers
\cite{yellowstone,oeis-yellowstone}.

The Enots--Wolley sequence, introduced by Shannon and Sloane as OEIS A336957,
reverses those two divisibility conditions \cite{oeis-ew}.  It is the
lexicographically earliest infinite sequence of distinct positive integers
beginning with $1,2$ in which each later term shares a prime divisor with its
predecessor and is coprime to the term two places earlier.  Its first terms are
\[
1,2,6,15,35,14,12,33,55,10,18,21,77,\ldots.
\]
For this rule, indefinite continuation is equivalent to requiring each new
term to introduce a prime divisor absent from its predecessor.  For example,
after $1,2$, choosing $4$ would satisfy the two visible divisibility conditions
but leave no possible next term; the continuation condition excludes $4$ and
permits $6$.  Section~\ref{sec:dynamics} records this equivalent local form for
the rank-capped sequences used below.

Every noninitial Enots--Wolley term has at least two distinct prime divisors.
It is conjectured that every such integer occurs in the unrestricted sequence
\cite{oeis-ew}.  There is a related binary sequence, OEIS A338833, in which
prime-support sets are replaced by the positions of the $1$'s in binary
expansions \cite{oeis-binary}.  Nichols's master's thesis proves that every
integer with at least two $1$'s in its binary expansion occurs in that sequence
and develops further properties of the original Enots--Wolley sequence
\cite{nichols-thesis}.  The binary analogue changes the numerical ordering of
supports.  The family studied here instead retains ordinary prime
factorization and numerical order while bounding the number of distinct prime
divisors of each selected term.

Write $\omega(m)$ for the number of distinct prime divisors of $m$.  For an
integer $K\ge2$, let $\EW[2..K]$ be the Enots--Wolley sequence with the additional
restriction
\[
2\le\omega(a_n)\le K\qquad(n\ge3).
\]
The lexicographically earliest rule is applied anew with this restriction; the
sequence is not obtained by deleting terms from unrestricted Enots--Wolley.
The special case $\EW[2..2]$ is recorded as the draft OEIS entry A399694
\cite{oeis-rank2}.

\begin{theorem}[Finite-cap surjectivity]\label{thm:main-intro}
For every fixed integer $K\ge2$, $\EW[2..K]$ is a permutation of
\[
\{1,2\}\cup\{m\ge1:2\le\omega(m)\le K\}.
\]
\end{theorem}

The qualification that $K$ is fixed is essential to the conclusion proved
here.  An unrestricted sequence and its rank-capped counterpart agree only
until the cap first excludes the unrestricted choice.  Although every fixed
allowed integer occurs for every sufficiently large fixed cap, its occurrence
index may depend on the cap.  The theorem gives no bound on these indices
uniform in $K$ and therefore does not imply the unrestricted conjecture.

\section{Outline of the proof}
\label{sec:proof-outline}

The first reduction is from integers to exact prime supports.  All integers
with the same support satisfy the same local conditions, so the lexicographically
earliest rule services each exact-support queue in increasing numerical order.
Section~\ref{sec:defective} calls a support \emph{saturated} when its entire
queue occurs and writes $\operatorname{Sat}(S)$ for this property.  By FIFO,
$\operatorname{Sat}(S)$ is equivalent to infinitely many services of $S$, so
surjectivity is equivalent to proving $\operatorname{Sat}(S)$ for every allowed
support.  Section~\ref{sec:dynamics} supplies the elementary ingredients needed
for this reduction: finite loss, recurrence of every prime, and the
impossibility of an eventual cover by any fixed finite set of primes.

Assume then that $\operatorname{Sat}(T)$ fails for some allowed support $T$.
Section~\ref{sec:defective} first proves \emph{eventual full return}: after a
finite cutoff, every transition from a term disjoint from $T$ back to a term
meeting $T$ must land in a term containing all of $T$.  After choosing such a
cutoff at a target-free term, the terms meeting $T$ occur in short
\emph{target episodes}.  Every episode begins full, has length at most two, and
contains at most one proper term.  Thus the number of proper terms after the
cutoff is at most the number of full terms there.  If $|T|=K$, full return is
already impossible, so the remaining case has $2\le |T|<K$.

Section~\ref{sec:envelope} puts the full and proper histories on a common
numerical scale.  Let $p=\min T$.  From every full value $v$ occurring after
the cutoff we construct an earlier proper value $w$ satisfying
\[
\frac vp<w<v.
\]
Consequently, if $H$ is the largest proper value seen by a given time, every
such full value seen by that time is below $pH$.  Hence the total number of full
terms is bounded by the complete arithmetic population of integers below $pH$
whose support contains $T$.

The decisive observation is made when the proper maximum increases.  Suppose
that $H$ is a new numerical proper-value record.  Its predecessor is full and
its two-back term is target-free.  Section~\ref{sec:record-shadow} shows that
every rank-$K$ integer below $H$ which contains exactly one prime of $T$ and
avoids the two-back support is an admissible smaller candidate, and therefore
must already have appeared.  These integers form a literal disjoint historical
population; no weighting or overlap correction is required.  Comparing this
population with the episode bound and the full-value envelope yields a fixed
upper bound on the difference between the proper record shadow and the entire
possible full population below $pH$.

Section~\ref{sec:counting} proves that this fixed upper bound is impossible for
large records.  Landau's fixed-order estimates imply that the record shadow has
order
\[
\frac{H(\log\log H)^{K-2}}{\log H},
\]
uniformly over the at-most-$K$ primes in the moving two-back support.  The
complete population of full integers below $pH$ has order
\[
\frac{H(\log\log H)^{K-|T|-1}}{\log H}.
\]
Since $|T|\ge2$, the latter is lower order.  Their difference therefore tends
to infinity along proper records, contradicting the fixed upper bound.
Section~\ref{sec:finish} concludes that no allowed support can be nonsaturated,
and hence every allowed exact-support queue is exhausted.

\section{Preliminaries and elementary dynamics}
\label{sec:dynamics}

For a positive integer $m$, let
\[
\supp(m)=\{\ell:\ell\text{ prime and }\ell\mid m\},
\qquad
\omega(m)=|\supp(m)|.
\]
We call $\omega(m)$ the \emph{rank} of $m$; thus a rank-$j$ integer is one
with exactly $j$ distinct prime divisors.  Fix $K\ge2$, and put
\[
\mathcal A_K=\{m\ge1:2\le\omega(m)\le K\}.
\]
The rank-capped sequence $(a_n)_{n\ge1}$ is defined by
\[
a_1=1,\qquad a_2=2,
\]
and, for $n\ge3$, by taking $a_n$ to be the least positive integer which has
not appeared earlier, belongs to $\mathcal A_K$, and satisfies
\begin{align}
\supp(a_n)\cap \supp(a_{n-1})&\ne\varnothing,
\label{eq:retention}\\
\supp(a_n)\cap \supp(a_{n-2})&=\varnothing,
\label{eq:lag-two}\\
\supp(a_n)\setminus \supp(a_{n-1})&\ne\varnothing.
\label{eq:novelty}
\end{align}
We refer to these as retention, lag-two disjointness, and novelty.

A prime support $S$ with $2\le |S|\le K$ is \emph{locally admissible} at a
selection if it satisfies the three support conditions corresponding to
\eqref{eq:retention}--\eqref{eq:novelty}.  Since these conditions depend only
on $S$, either every unused integer with exact support $S$ is locally
admissible or none is.  Write
\[
\mathcal Q_S=\{m\ge1:\supp(m)=S\}
\]
for the \emph{exact-support queue} of $S$, ordered increasingly.

The following observation makes precise why the local novelty condition is
exactly the condition needed in the lexicographically earliest formulation.

\begin{continuationlemma}[Infinite continuation criterion]
Let $a_1,\ldots,a_n$ be a finite admissible prefix of the rank-capped problem:
it begins with the prescribed terms $1,2$, every later term lies in
$\mathcal A_K$, the terms are distinct, and the two Enots--Wolley divisibility
conditions hold at every applicable index.  Let $b$ be an unused member of
$\mathcal A_K$ satisfying retention with $a_n$ and lag-two disjointness from
$a_{n-1}$.  Then the extended prefix
\[
a_1,\ldots,a_n,b
\]
admits an infinite continuation satisfying the same conditions if and only if
\[
\supp(b)\setminus\supp(a_n)\ne\varnothing.
\]
\end{continuationlemma}

\begin{proof}
If $\supp(b)\subseteq\supp(a_n)$, then every possible successor sharing a
prime with $b$ also shares a prime with the two-back term $a_n$.  Retention and
lag-two disjointness therefore cannot both hold, so no continuation exists.

Conversely, choose
\[
r\in\supp(b)\setminus\supp(a_n).
\]
Choose distinct primes $Q_1,Q_2,\ldots$ which occur in none of
$a_1,\ldots,a_n,b$ and are different from $r$.  Append
\[
rQ_1,\quad Q_1Q_2,\quad Q_2Q_3,\quad\ldots.
\]
The first appended term retains $r$ from $b$, avoids $a_n$, and introduces
$Q_1$.  The next term $Q_1Q_2$ is disjoint from $b$ because both $Q_1$ and
$Q_2$ were chosen outside $\supp(b)$.  Thereafter each term retains the newest
prime from its predecessor, avoids the term two places back, and introduces a
fresh prime.  All appended terms have rank two and are distinct.  Hence this
gives an infinite continuation.
\end{proof}

Thus, among unused allowed integers satisfying retention and lag-two
disjointness, the extendible choices are exactly those satisfying novelty.
The local rule above is therefore precisely the lexicographically earliest
infinite sequence beginning with $1,2$ and satisfying the rank restriction
together with the two Enots--Wolley divisibility conditions.

The sequence never gets stuck.  At the first greedy selection, $2Q$ is a
candidate for any prime $Q\ne2$.  Thereafter novelty of the preceding term
supplies an active prime
\[
r\in \supp(a_{n-1})\setminus \supp(a_{n-2}),
\]
and pairing $r$ with any globally unseen prime $Q$ gives an unused rank-two
candidate $rQ$.

\begin{lemma}[FIFO service and finite loss]\label{lem:fifo-loss}
Let $S$ be an allowed support.
\begin{enumerate}
\item The selected members of $\mathcal Q_S$ occur in increasing numerical
order.  In particular, if $S$ is selected infinitely often, then every member
of $\mathcal Q_S$ is eventually selected.
\item If $S$ is selected only finitely often, then after some finite time it is
never locally admissible.
\end{enumerate}
\end{lemma}

\begin{proof}
Whenever $S$ is locally admissible, the least unused member of $\mathcal Q_S$
is a candidate, so services of $S$ occur in increasing order.  This proves the
first assertion.

For the second, suppose service stops and let $h$ be the resulting least unused
member of $\mathcal Q_S$.  On every later occasion when $S$ is locally
admissible, $h$ is a candidate.  If it loses, the selected winner is a new
integer strictly below $h$.  Since there are only finitely many such integers,
this can happen only finitely often.
\end{proof}

We next record the two recurrence facts used in the main argument.

\begin{lemma}[Prime debuts are exact pairs]\label{lem:debut-pair}
If a globally new prime $Q$ first occurs in a selected term $V$, then
\[
V=cQ
\]
for a unique prime $c$ retained from the predecessor.
\end{lemma}

\begin{proof}
Choose $c\in \supp(V)\cap \supp(a_{n-1})$.  Because $V$ satisfies lag-two
disjointness, $c\notin \supp(a_{n-2})$.  Since $Q$ is globally new, the integer
$cQ$ is unused; it retains $c$, avoids $a_{n-2}$, introduces $Q$, and has rank
two.  Thus $cQ$ is a candidate at the same selection, so greediness gives
$V\le cQ$.  Conversely, $cQ\mid V$, hence $cQ\le V$.  Therefore $V=cQ$; in
particular, $c$ is the unique old prime dividing $V$.
\end{proof}

\begin{theorem}[Prime recurrence]\label{thm:prime-recurrence}
Every prime occurs in infinitely many selected terms.
\end{theorem}

\begin{proof}
First observe that infinitely many primes debut.  Otherwise let $F$ be the
finite set of all primes that ever appear in the sequence, and choose a prime
$Q\notin F$.  Then $Q$ is globally unseen.  At every subsequent selection, an
active prime $r\in F$ gives the unused candidate $rQ$, so every subsequent
winner is at most $Q\max F$.  This is impossible because the sequence is
injective and has infinitely many subsequent terms.

Now fix a prime $p$ and suppose, toward a contradiction, that there is an index
$N_p$ such that
\[
p\nmid a_n\qquad(n\ge N_p).
\]
Choose a debut prime $Q>p$ at an index large enough that the entire packet
\[
A\longrightarrow cQ\longrightarrow B\longrightarrow C
\]
occurs at indices at least $N_p$.  At the selection of $B$, the prime $c$ is
blocked by the two-back term $A$, so retention from $cQ$ must occur through
$Q$; hence $Q\mid B$.  The integer $Qp$ is unused and is also a candidate: it
retains $Q$, avoids $A$, and introduces $p$.  Since $p$ does not occur at or
after $N_p$, equality cannot occur, and therefore $B<Qp$.  Writing $B=Qm$, we
obtain
\[
1<m<p.
\]
If $p=2$, this is already impossible.  Otherwise choose a prime divisor
$u\mid m$.  Then $u<p<Q$, and $u\ne c$: the prime $c$ divides $A$, while $B$
is disjoint from $A$.  Also $c\ne p$, because $c\mid A$ and $A$ occurs at an
index at least $N_p$.  Hence the pair $\{u,p\}$ is disjoint from the support of
$cQ$.

At the selection of $C$, the pair support $\{u,p\}$ retains $u$ from $B$,
avoids the two-back term $cQ$, and introduces $p$.  Because $p$ never appears
at or after $N_p$, the least unused member $h_{u,p}$ of this pair queue is fixed
from that index onward.  It is a candidate at this selection, while the actual
winner $C$ does not contain $p$.  Hence
\[
C<h_{u,p}.
\]
There are only finitely many possible primes $u<p$.  Passing to an infinite
subsequence of debut packets, we may therefore fix $u=u_0$.  The associated
selected values $C$ are distinct by injectivity and all satisfy
$C<h_{u_0,p}$, a contradiction.
\end{proof}

\begin{lemma}[No finite eventual prime cover]\label{lem:no-finite-cover}
For every finite set $T$ of primes, infinitely many selected terms are disjoint
from $T$.
\end{lemma}

\begin{proof}
Suppose instead that there is an index $N_T$ such that
\[
\supp(a_n)\cap T\ne\varnothing
\qquad(n\ge N_T).
\]
Choose a prime $x\notin T$ and distinct primes
\[
s_1,\ldots,s_{2K+1}\notin T\cup\{x\}.
\]
Each pair support $\{x,s_i\}$ is disjoint from $T$, so it can be selected only
finitely often; by Lemma~\ref{lem:fifo-loss}, each is eventually never locally
admissible.

By Theorem~\ref{thm:prime-recurrence}, the prime $x$ occurs infinitely often.
Every consecutive run of $x$-containing terms has length at most two, because
lag-two disjointness forbids $x$ at positions two apart.  Hence infinitely many
occurrences of $x$ force infinitely many run starts, and every run start is an
index $n$ with
\[
x\in \supp(a_n)\setminus \supp(a_{n-1}).
\]
Choose such an $n$ after all pair-support cutoffs.  At the selection of
$a_{n+1}$, the support $\{x,s_i\}$ retains $x$, avoids $a_{n-1}$, and has
novelty exactly when $s_i\notin \supp(a_n)$.  Hence it is locally admissible
whenever
\[
s_i\notin \supp(a_{n-1})\cup \supp(a_n).
\]
The displayed union contains at most $2K$ primes, so it cannot contain all
$2K+1$ choices $s_i$.  At least one forbidden pair support is therefore
locally admissible, a contradiction.
\end{proof}

\section{Nonsaturated supports and full return}
\label{sec:defective}

By Lemma~\ref{lem:fifo-loss}, the selected elements of every allowed exact-
support queue occur in increasing numerical order.  We say that the queue
$\mathcal Q_S$ is \emph{saturated} if every one of its elements appears as a
term of the sequence, and write
\[
\operatorname{Sat}(S)
\]
for the proposition that $\mathcal Q_S$ is saturated.  Since every queue is
infinite, Lemma~\ref{lem:fifo-loss} shows that $\operatorname{Sat}(S)$ is
equivalent to saying that the exact support $S$ is selected infinitely often.
Thus Theorem~\ref{thm:main-intro} is equivalent to proving
$\operatorname{Sat}(S)$ for every allowed support $S$.

Suppose that $\operatorname{Sat}(S)$ fails.  Then $S$ is selected only finitely
often, and Lemma~\ref{lem:fifo-loss} gives an index after which $S$ is never
locally admissible.  This immediately forces a useful return property.

\begin{lemma}[Eventual full return]\label{lem:full-entry}
If $\operatorname{Sat}(S)$ fails, then there is an index after which every
transition
\[
B\longrightarrow F
\]
from a term $B$ satisfying $\supp(B)\cap S=\varnothing$ to a term $F$ satisfying
$\supp(F)\cap S\ne\varnothing$ has
\[
S\subseteq\supp(F).
\]
\end{lemma}

\begin{proof}
Choose an index after which $S$ is never locally admissible, and consider such a
transition beyond that index.  At the selection immediately following $F$, the
support $S$ overlaps the predecessor $F$ and avoids the two-back term $B$.  If
$F$ omitted some prime of $S$, then $S$ would also supply novelty.  Thus $S$
would be locally admissible, a contradiction.  Hence $F$ contains every prime
of $S$.
\end{proof}

When $\operatorname{Sat}(S)$ fails, call an index $N$ a \emph{full-return
cutoff for $S$} if no exact-$S$ term is selected after $N$, $S$ is never
locally admissible after $N$, and every transition $B\to F$ ending after $N$
from an $S$-free term to a term meeting $S$ has $S\subseteq\supp(F)$.  Such
cutoffs exist by Lemma~\ref{lem:fifo-loss} and Lemma~\ref{lem:full-entry}, and
every larger index is again a full-return cutoff.

Fix now an allowed support
\[
T=\{t_1,\ldots,t_d\},
\qquad
2\le d\le K,
\]
and suppose, toward a contradiction, that
\[
\neg\operatorname{Sat}(T).
\]
Choose a full-return cutoff for $T$.  Lemma~\ref{lem:no-finite-cover} supplies
a later term disjoint from $T$; let $n_0$ be the index of one such term.  Since
every larger index remains a full-return cutoff, we may and do take $n_0$
itself as our fixed cutoff.  Thus
\begin{equation}
\begin{gathered}
T\text{ is never locally admissible after }n_0,\\
\supp(a_{n_0})\cap T=\varnothing,
\end{gathered}
\label{eq:tail-cutoff}
\end{equation}
and no term after $n_0$ has exact support $T$.  We call the sequence of terms
\[
a_{n_0+1},a_{n_0+2},\ldots
\]
the \emph{$T$-tail}, or simply the \emph{tail} while $T$ is fixed.

Relative to $T$, call a term \emph{free} if its support is disjoint from $T$,
\emph{full} if its support contains all of $T$, and \emph{proper} if it meets
$T$ but does not contain all of $T$.  A \emph{target episode} is a maximal
block of consecutive nonfree terms in the tail.

By the full-return property, every target episode begins with a full term.

\begin{lemma}[Episode grammar]\label{lem:episode-grammar}
Every target episode in the tail has length at most two, and between two
successive episodes there are at least two consecutive free terms.  Every
proper term is therefore the second term of a unique episode whose first term
is full.
\end{lemma}

\begin{proof}
Let an episode begin at index $n$.  Its starter $a_n$ is full by the
full-return property.  Lag-two disjointness makes $a_{n+2}$ disjoint from every
prime of $T$, so $a_{n+2}$ is free and the episode has length at most two.

If the episode has length one, then $a_{n+1}$ is free; lag-two disjointness with
the full term $a_n$ also makes $a_{n+2}$ free.  If the episode has length two,
then $a_{n+2}$ is free.  Were $a_{n+3}$ the start of the next episode, the
full-return property would make it full; but it would then share a target prime
with the nonfree two-back term $a_{n+1}$, contradicting lag-two disjointness.
Hence $a_{n+3}$ is free as well.

The final assertion follows because the first term of each episode is full and
there is at most one second term.
\end{proof}

Every prime of $T$ recurs by Theorem~\ref{thm:prime-recurrence}.  Consequently
the tail contains infinitely many nonfree terms.  Since every target episode
has length at most two, there are infinitely many target episodes and hence
infinitely many full starters.

For $N>n_0$, let $N_{\mathrm{prop}}(N)$ and $N_{\mathrm{full}}(N)$ denote the
numbers of proper and full terms, respectively, among
\[
a_{n_0+1},\ldots,a_N.
\]
Pair each proper term with the full starter immediately preceding it in its
episode.  This pairing is injective, and hence
\begin{equation}
\boxed{N_{\mathrm{prop}}(N)\le N_{\mathrm{full}}(N).}
\label{eq:episode-bound}
\end{equation}

The top-rank case is already impossible.  If $d=K$, every full support has rank
at least $K$ and contains $T$, so it must equal $T$.  But exact support $T$
does not occur in the tail, whereas there are infinitely many full starters.
Consequently every nonsaturated support in a putative counterexample must
satisfy
\begin{equation}
\boxed{2\le d<K.}
\label{eq:defective-rank-range}
\end{equation}
Thus the remainder of the argument concerns only $K\ge3$; when $K=2$, every
allowed support has size $K$ and the preceding contradiction already completes
the proof.  In particular, every full term in the remaining case contains at
least one prime outside $T$.

\section{A near-height proper replacement}
\label{sec:envelope}

Continue under the nonsaturation hypothesis
$\neg\operatorname{Sat}(T)$ of Section~\ref{sec:defective}, so
$2\le d=|T|<K$, and put
\[
p=\min T.
\]
For a full integer $v$, separate its target-prime powers from its outside part:
\begin{equation}
\begin{aligned}
v&=c_T(v)z(v),
& c_T(v)&=\prod_{t\in T}t^{e_t(v)},\qquad e_t(v)\ge1,\\
&& \gcd\!\left(z(v),\prod_{t\in T}t\right)&=1.
\end{aligned}
\label{eq:target-outside-factorization}
\end{equation}
Because no term in the tail has exact support $T$, every full tail term has
$z(v)>1$.

Let $p^j$ be the largest power of $p$ strictly smaller than $c_T(v)$.  Since
$|T|\ge2$, the integer $c_T(v)$ is divisible by $p$ and by another prime, so
$c_T(v)>p$ and therefore $j\ge1$.  By maximality,
\[
p^j<c_T(v)\le p^{j+1}.
\]
Because $c_T(v)$ has a prime divisor different from $p$, it is not a pure power
of $p$.  Therefore the second inequality is also strict:
$c_T(v)<p^{j+1}$.  Dividing by $p$ gives
\begin{equation}
\frac{c_T(v)}p<p^j<c_T(v).
\label{eq:p-power-sandwich}
\end{equation}
Define
\begin{equation}
w(v)=p^jz(v).
\label{eq:proper-replacement}
\end{equation}
Then
\[
\supp(w(v))=\{p\}\cup\bigl(\supp(v)\setminus T\bigr).
\]
If $r=|\supp(v)\setminus T|$, then $1\le r\le K-d$, so
\[
2\le\omega(w(v))=r+1\le K-d+1\le K-1.
\]
Thus $w(v)$ is an allowed proper integer.

\begin{lemma}[Near-height proper replacement]\label{lem:near-height}
For every full term $v$ in the tail, the integer $w(v)$ of
\eqref{eq:proper-replacement} was selected before $v$.  Moreover
\begin{equation}
\boxed{\frac vp<w(v)<v.}
\label{eq:near-height}
\end{equation}
\end{lemma}

\begin{proof}
Because $\supp(w(v))\subseteq \supp(v)$, lag-two disjointness of $v$ implies
lag-two disjointness of $w(v)$.  It remains to verify retention and novelty.

If $v$ is an episode starter, its predecessor is free.  Retention of $v$ must
therefore occur through a prime outside $T$, and every outside prime of $v$
remains in $w(v)$.  The prime $p$, absent from the free predecessor, supplies
novelty for $w(v)$.

If $v$ is the second term of an episode and is full, then its predecessor is
full.  The replacement $w(v)$ retains $p$.  Since every target prime is already
present in the predecessor, novelty of $v$ must be supplied by an outside prime;
that outside prime is again retained in $w(v)$.

Thus $w(v)$ is locally admissible at the selection of $v$.  By
\eqref{eq:p-power-sandwich}, it is numerically smaller than $v$, so greedy
minimality forces it to have been selected earlier.  Multiplying
\eqref{eq:p-power-sandwich} by $z(v)$ gives \eqref{eq:near-height}.
\end{proof}

The two cases in the proof of Lemma~\ref{lem:near-height} are summarized in
Table~\ref{tab:near-height-cases}.

\begin{table}[ht]
\centering
\caption{Retention and novelty for the replacement $w(v)$ in
Lemma~\ref{lem:near-height}.}
\label{tab:near-height-cases}
\begin{tabular}{lll}
\toprule
Predecessor of $v$ & Retention for $w(v)$ & Novelty for $w(v)$ \\
\midrule
free & an outside carrier of $v$ & target prime $p$ \\
full & target prime $p$ & an outside novelty of $v$ \\
\bottomrule
\end{tabular}
\end{table}

For an index $N>n_0$, let
\begin{equation}
H(N)=\max\{a_n:n\le N\text{ and }a_n\text{ is proper}\},
\label{eq:proper-height}
\end{equation}
whenever the set is nonempty.  Thus $H(N)$ is the running maximum of the
selected proper values.  In particular, if $v=a_n$ is a full term with
$n_0<n\le N$, then Lemma~\ref{lem:near-height} shows that its replacement
$w(v)$ from \eqref{eq:proper-replacement} was selected before $v$, and hence
\[
w(v)\le H(N).
\]
Combining this with \eqref{eq:near-height} gives the common numerical envelope
\begin{equation}
\boxed{v<pH(N)}
\label{eq:full-envelope}
\end{equation}
for every full term $v=a_n$ with $n_0<n\le N$.

Define
\begin{equation}
\Full_{K,T}(X)
=
\#\{m<X:T\subseteq \supp(m),\ \omega(m)\le K\}.
\label{eq:full-population}
\end{equation}
Since the sequence is injective, \eqref{eq:full-envelope} implies
\begin{equation}
\boxed{N_{\mathrm{full}}(N)\le \Full_{K,T}(pH(N)).}
\label{eq:full-count-envelope}
\end{equation}

There are infinitely many distinct full starters in the tail, so their
numerical values are unbounded.  By \eqref{eq:near-height}, the selected proper
values are also unbounded.  Only finitely many proper values occur at or before
the fixed cutoff $n_0$, so the unbounded tail produces arbitrarily large global
proper-value records after $n_0$.

\section{The exact shadow at a proper record}
\label{sec:record-shadow}

Let $H=a_n$ be a global proper-value record with $n>n_0$.  Here
\emph{record} always refers to numerical value: the term $H$ is proper, and
$H$ exceeds every proper value selected at an earlier index.  By
Lemma~\ref{lem:episode-grammar}, a proper term is the second term of its target
episode, so the local picture is
\[
A_{\mathrm{free}}
\longrightarrow
V_{\mathrm{full}}
\longrightarrow
H_{\mathrm{proper}}.
\]
Write $B=\supp(A)$.  Then
\begin{equation}
B\cap T=\varnothing,
\qquad
|B|\le K.
\label{eq:blocker-conditions}
\end{equation}

Define the proper-record shadow
\begin{equation}
\Shadow_{K,T}(H;B)
=
\left\{m<H:
\begin{gathered}
\omega(m)=K,\\
|\supp(m)\cap T|=1,\\
\supp(m)\cap B=\varnothing
\end{gathered}
\right\},
\label{eq:record-shadow-set}
\end{equation}
and let
\[
R_{K,T}(H;B)=|\Shadow_{K,T}(H;B)|.
\]

\begin{lemma}[Automatic novelty at a proper record]
\label{lem:automatic-novelty}
Every member of the record shadow \eqref{eq:record-shadow-set} was selected at
an index strictly smaller than $n$.
\end{lemma}

\begin{proof}
Fix $m\in\Shadow_{K,T}(H;B)$.  Its unique target prime belongs to the full
predecessor $V$, so $m$ satisfies retention.  The condition
$\supp(m)\cap B=\varnothing$ gives lag-two disjointness.

It remains only to verify novelty.  If novelty failed, then
\[
\supp(m)\subseteq \supp(V).
\]
But $\omega(m)=K$ while $\omega(V)\le K$, so this containment would force
$\supp(m)=\supp(V)$.  That is impossible: $V$ contains all $d\ge2$ primes of
$T$, whereas $m$ contains exactly one.  Hence $m$ introduces at least one
prime relative to $V$.

Therefore $m$ is locally admissible at the selection of $H=a_n$.  If it were
unused at time $n$, it would be a smaller candidate than the observed winner
$H$.  Greedy minimality forces $m$ to have been selected at an earlier index.
\end{proof}

The same novelty argument can be expressed as a rank count: every support in
\eqref{eq:record-shadow-set} has $K-1$ primes outside $T$, whereas the full
predecessor has at most $K-d\le K-2$ outside primes.  Thus every shadow support
contains an outside prime absent from the predecessor.

For reference, the essential incidence pattern is
\[
\begin{array}{c|ccc}
& A & V & m\\ \hline
\text{unique target prime of }m &0&1&1\\
\text{other target primes}&0&1&0\\
\text{some outside novelty}&0&0&1\\
\text{remaining outside primes}&0&*&1
\end{array}
\]
where the zeros in the first column follow from $A$ being free together with
$\supp(m)\cap B=\varnothing$.

Let $C_T$ denote the number of proper terms selected at or before the fixed
cutoff $n_0$.  Lemma~\ref{lem:automatic-novelty} implies
\begin{equation}
R_{K,T}(H;B)
\le C_T+N_{\mathrm{prop}}(n),
\label{eq:shadow-vs-proper-count}
\end{equation}
because every shadow integer is a proper historical term.  At the record time,
$H(n)=H$.  Combining \eqref{eq:episode-bound},
\eqref{eq:full-count-envelope}, and \eqref{eq:shadow-vs-proper-count} yields
the necessary inequality
\begin{equation}
\boxed{
R_{K,T}(H;B)-\Full_{K,T}(pH)\le C_T.
}
\label{eq:record-necessary}
\end{equation}
The right-hand side depends only on the fixed opening segment; in particular it
is independent of the record height and of the actual two-back support $B$.

\section{Asymptotic size of the record shadow}
\label{sec:counting}

We compare the two terms in \eqref{eq:record-necessary}.  The analytic input is
the classical fixed-order theorem of Landau~\cite[pp.~203--211]{landau-handbuch};
see also Tenenbaum~\cite[Chapter~II.6]{tenenbaum2015}.  For every fixed integer
$j\ge1$,
\begin{equation}
\pi_{=j}(X):=\#\{n<X:\omega(n)=j\}
\sim
\frac{X}{\log X}\frac{(\log\log X)^{j-1}}{(j-1)!}.
\label{eq:landau}
\end{equation}
We use Vinogradov notation $A\ll_\alpha B$ to mean
$A\le C_\alpha B$ for some constant $C_\alpha>0$ depending only on the
indicated parameter(s) $\alpha$; with no subscript, the constant is absolute.
We shall also use the standard fixed-order consequence
\begin{equation}
\pi_{\le j}(X):=\#\{n<X:\omega(n)\le j\}
\ll_j
\frac{X(\log\log X)^{j-1}}{\log X}
\label{eq:landau-upper}
\end{equation}
for large $X$.  Since every lower positive fixed order is smaller by at least
one power of $\log\log X$, $\pi_{\le j}(X)$ has the same leading term as
$\pi_{=j}(X)$.  The sole rank-zero integer $n=1$ contributes only $1$ and is
negligible.

We will repeatedly apply \eqref{eq:landau-upper} at a height varying between
$\sqrt X$ and $X$.  The following uniform form records the required
comparability once and for all: for fixed $j\ge1$,
\begin{equation}
\sup_{\sqrt X\le Y\le X}
\frac{\pi_{\le j}(Y)}{Y}
\ll_j
\frac{(\log\log X)^{j-1}}{\log X}.
\label{eq:landau-window}
\end{equation}
Indeed, on this interval $\log Y\asymp\log X$ and
$\log\log Y\ll\log\log X$.

For a nonempty finite prime set $S$, put
\begin{equation}
\kappa(S)=\prod_{s\in S}\frac1{s-1}.
\label{eq:anchor-weight}
\end{equation}

\begin{lemma}[Removing finitely many primes]\label{lem:forbidden-primes}
Fix $j\ge1$ and a finite set $Q$ of primes.  Then
\[
\#\{n<X:\omega(n)=j,\ \supp(n)\cap Q=\varnothing\}
\sim \pi_{=j}(X).
\]
The same main term holds with $\omega(n)\le j$.
\end{lemma}

\begin{proof}
Fix $q\in Q$ and let
\[
E_{j,q}(X)=\{n<X:\omega(n)=j,\ q\mid n\}
\]
be the exceptional rank-$j$ integers excluded by the prime $q$.  Every
$n\in E_{j,q}(X)$ has a unique factorization
\[
n=q^av,
\qquad a\ge1,
\qquad q\nmid v.
\]
If $j=1$, only powers of $q$ occur, giving $O(\log X)$ such exceptional
integers.  Since
\[
\frac{O(\log X)}{\pi_{=1}(X)}
=O\!\left(\frac{(\log X)^2}{X}\right)\longrightarrow0,
\]
there is nothing further to prove in this case.

Suppose now that $j\ge2$.  Since $q$ accounts for one of the $j$ distinct
prime divisors of $n=q^av$, every exceptional integer has
\[
\omega(v)=j-1.
\]
Thus, for each fixed exponent $a$, counting exceptional integers
$n=q^av<X$ reduces to counting rank-$(j-1)$ integers
\[
v<\frac{X}{q^a}.
\]
We consider separately the ranges
\[
q^a\le\sqrt X
\qquad\text{and}\qquad
q^a>\sqrt X.
\]
The first range is where the uniform Landau estimate
\eqref{eq:landau-window} applies; in the second range the cofactor $v$ is
already smaller than $\sqrt X$, so a crude bound suffices.

First suppose $q^a\le\sqrt X$.  For this fixed exponent $a$, put
\[
Y=\frac{X}{q^a}.
\]
Then
\[
\sqrt X\le Y\le X.
\]
Equation~\eqref{eq:landau-window}, with $j-1$ in place of $j$, therefore gives
\[
\pi_{\le j-1}(Y)
\ll_j
Y\frac{(\log\log X)^{j-2}}{\log X}
=
\frac{X}{q^a}\frac{(\log\log X)^{j-2}}{\log X}.
\]
The possible cofactors $v$ have exactly $j-1$ distinct prime divisors, so
their number is at most $\pi_{\le j-1}(Y)$.  Hence the number of exceptional
integers corresponding to this particular exponent $a$ is at most
\[
O_j\!\left(
\frac{X}{q^a}\frac{(\log\log X)^{j-2}}{\log X}
\right).
\]
Summing over all exponents with $q^a\le\sqrt X$, and enlarging the geometric
sum if necessary, gives
\[
\begin{aligned}
\#\{n\in E_{j,q}(X):q^a\le\sqrt X\}
&\ll_j
\frac{X(\log\log X)^{j-2}}{\log X}
\sum_{a\ge1}\frac1{q^a}\\
&=
O_{j,q}\!\left(
\frac{X(\log\log X)^{j-2}}{\log X}
\right).
\end{aligned}
\]
Comparing this with \eqref{eq:landau}, the exceptional count is smaller than
$\pi_{=j}(X)$ by a factor $O_{j,q}(1/\log\log X)$, and hence is
$o(\pi_{=j}(X))$.

Now suppose $q^a>\sqrt X$.  Then
\[
v<\frac{X}{q^a}<\sqrt X.
\]
There are only $O(\log X)$ possible exponents $a$, and for each there are
fewer than $\sqrt X$ possible integers $v$.  Thus this range contributes at
most
\[
O(\sqrt X\log X)=o(\pi_{=j}(X)).
\]
Therefore $|E_{j,q}(X)|=o(\pi_{=j}(X))$ for every fixed $q\in Q$.
Since $Q$ is finite, the union of these exceptional sets is still
$o(\pi_{=j}(X))$, proving the exact-rank claim.  The cumulative version follows
by summing the fixed positive orders, with the single integer $n=1$
contributing only $O(1)$.
\end{proof}

\begin{lemma}[Fixed anchor set]\label{lem:fixed-anchors}
Let $S$ be a fixed set of $s$ primes and let $K>s$.  Then
\begin{equation}
\#\{m<X:S\subseteq \supp(m),\ \omega(m)=K\}
\sim
\frac{\kappa(S)}{(K-s-1)!}
\frac{X(\log\log X)^{K-s-1}}{\log X}.
\label{eq:fixed-anchor-asymptotic}
\end{equation}
The same main term holds with $\omega(m)\le K$.
\end{lemma}

\begin{proof}
Put $j=K-s\ge1$ and
\[
\Lambda_j(X)=\frac{X(\log\log X)^{j-1}}{\log X}.
\]
For an exponent vector $\mathbf b=(b_t)_{t\in S}$ with every $b_t\ge1$, define
\[
u(\mathbf b)=\prod_{t\in S}t^{b_t}.
\]
Every integer counted in the exact-rank population has a unique factorization
\[
m=u(\mathbf b)v,
\qquad
\supp(v)\cap S=\varnothing,
\qquad
\omega(v)=j.
\]
For brevity, write
\[
N_{S,K}(X)
=
\#\{m<X:S\subseteq\supp(m),\ \omega(m)=K\}.
\]

For every fixed $u>0$,
\begin{equation}
\Lambda_j(X/u)\sim \frac1u\Lambda_j(X),
\label{eq:lambda-scaling}
\end{equation}
because $\log(X/u)\sim\log X$ and
$\log\log(X/u)\sim\log\log X$.  Fix $M$.  For the finitely many exponent
vectors satisfying
\[
u(\mathbf b)\le M,
\]
Lemma~\ref{lem:forbidden-primes}, \eqref{eq:landau}, and
\eqref{eq:lambda-scaling} therefore apply term by term.  On the complementary
range
\[
M<u(\mathbf b)\le\sqrt X,
\]
the uniform bound \eqref{eq:landau-window} gives a total at most a constant
multiple of
\[
\Lambda_j(X)
\sum_{u(\mathbf b)>M}\frac1{u(\mathbf b)}.
\]
For $u(\mathbf b)>\sqrt X$, there are $O_S((\log X)^s)$ possible exponent
vectors and fewer than $\sqrt X$ possible cofactors for each, giving
$O_S(\sqrt X(\log X)^s)=o(\Lambda_j(X))$.  Consequently
\begin{equation}
\begin{aligned}
\frac{N_{S,K}(X)}{\Lambda_j(X)}
&=
\frac1{(j-1)!}
\sum_{u(\mathbf b)\le M}\frac1{u(\mathbf b)}
+o_M(1)\\
&\quad+
O_S\!\left(
\sum_{u(\mathbf b)>M}\frac1{u(\mathbf b)}
\right)
+o(1).
\end{aligned}
\label{eq:anchor-truncation}
\end{equation}
Finally,
\[
\sum_{\mathbf b}\frac1{u(\mathbf b)}
=
\prod_{t\in S}\sum_{b\ge1}t^{-b}
=
\prod_{t\in S}\frac1{t-1}
=\kappa(S).
\]
First let $X\to\infty$ with $M$ fixed in \eqref{eq:anchor-truncation}, then let
$M\to\infty$.  This proves \eqref{eq:fixed-anchor-asymptotic}.

For the cumulative condition $\omega(m)\le K$, each exact rank $s<r<K$ is
handled by the same argument with $K$ replaced by $r$, giving
\[
\begin{aligned}
&\#\{m<X:S\subseteq\supp(m),\ \omega(m)=r\}\\
&\qquad\ll_{S,r}
\frac{X(\log\log X)^{r-s-1}}{\log X}
=o\!\left(
\frac{X(\log\log X)^{K-s-1}}{\log X}
\right).
\end{aligned}
\]
The remaining rank-$s$ integers are supported entirely on $S$; there are only
$O_S((\log X)^s)$ of them, which is also lower order than the main term in
\eqref{eq:fixed-anchor-asymptotic}.  Hence the cumulative population has the
same main term as exact rank $K$.
\end{proof}

The two-back support $B$ varies with the record.  We therefore need a bound
uniform in the identity of a moving forbidden prime.

\begin{lemma}[Uniform cost of one moving forbidden prime]
\label{lem:moving-mask}
Fix $K\ge3$ and a prime $t$.  Uniformly over primes $r\ne t$,
\begin{equation}
\#\{m<X:t,r\mid m,\ \omega(m)\le K\}
\ll_{K,t}
\frac{X(\log\log X)^{K-3}}{\log X}
+\sqrt X(\log X)^2.
\label{eq:moving-mask-bound}
\end{equation}
\end{lemma}

\begin{proof}
Factor the exact powers of $t$ and $r$:
\[
m=t^ar^bv,
\qquad a,b\ge1,
\qquad \gcd(v,tr)=1,
\qquad \omega(v)\le K-2.
\]
Dropping the coprimality condition can only increase the count.  If
$t^ar^b\le\sqrt X$, put $Y=X/(t^ar^b)$.  Then
$\sqrt X\le Y\le X$, so \eqref{eq:landau-window} gives
\[
\#\{v<Y:\omega(v)\le K-2\}
\ll_K
\frac{Y(\log\log X)^{K-3}}{\log X}.
\]
Summing over these exponent pairs and using
\[
\sum_{a,b\ge1}\frac1{t^ar^b}
=
\frac1{(t-1)(r-1)}
\le\frac1{t-1}
\]
gives the first term of \eqref{eq:moving-mask-bound}, uniformly in $r$.
For the remaining exponent pairs,
\[
t^ar^b>\sqrt X,
\]
there are $O((\log X)^2)$ possibilities, uniformly in $r$, and fewer than
$\sqrt X$ possible cofactors for each.  This gives the second term.
\end{proof}

Define
\begin{equation}
f_K(H)=\frac{H(\log\log H)^{K-2}}{\log H},
\qquad
c_{K,T}=\frac1{(K-2)!}\sum_{t\in T}\frac1{t-1}>0.
\label{eq:shadow-main-scale}
\end{equation}

\begin{theorem}[Exact proper-record surplus]\label{thm:record-surplus}
Fix $K$ and a prime set $T$ with $2\le d=|T|<K$, and put $p=\min T$.
As $H\to\infty$,
\begin{equation}
R_{K,T}(H;B)
=
\bigl(c_{K,T}+o(1)\bigr)f_K(H),
\label{eq:shadow-asymptotic}
\end{equation}
where the $o(1)$ is uniform over all prime sets $B$ satisfying
\[
B\cap T=\varnothing,
\qquad
|B|\le K.
\]
Here $K$ and $T$ are fixed throughout: all implied constants and asymptotic
thresholds may depend on them.  The asserted uniformity is only in the moving
blocker $B$.
Moreover,
\begin{equation}
\Full_{K,T}(pH)
\sim
\frac{p\,\kappa(T)}{(K-d-1)!}
\frac{H(\log\log H)^{K-d-1}}{\log H}.
\label{eq:full-asymptotic}
\end{equation}
Consequently, uniformly over all such $B$,
\begin{equation}
\boxed{
R_{K,T}(H;B)-\Full_{K,T}(pH)
=
\bigl(c_{K,T}+o(1)\bigr)f_K(H)
\longrightarrow+\infty.
}
\label{eq:surplus-asymptotic}
\end{equation}
\end{theorem}

\begin{proof}
Fix $t\in T$.  Lemma~\ref{lem:fixed-anchors} with anchor set $\{t\}$ gives
\[
\#\{m<H:t\mid m,\ \omega(m)=K\}
=
\left(\frac1{(t-1)(K-2)!}+o(1)\right)f_K(H).
\]
To obtain the $t$-component of the record shadow, exclude integers divisible by
one of the other $d-1$ target primes and by any prime of $B$.
Lemma~\ref{lem:moving-mask} bounds the loss caused by each excluded prime,
uniformly in that prime, by
\[
O_{K,T}\!\left(
\frac{H(\log\log H)^{K-3}}{\log H}
+\sqrt H(\log H)^2
\right)
=o(f_K(H)).
\]
There are at most $d-1+K$ excluded primes, so the total loss is still
$o(f_K(H))$ uniformly over $B$.  The surviving components for different
$t\in T$ are disjoint because each surviving integer contains exactly one
prime of $T$.  Summing over $t$ proves \eqref{eq:shadow-asymptotic}.

For the full population, apply Lemma~\ref{lem:fixed-anchors} to the fixed anchor
set $T$ at height $pH$.  Since $p$ is fixed,
\[
\log(pH)\sim\log H,
\qquad
\log\log(pH)\sim\log\log H,
\]
which gives \eqref{eq:full-asymptotic}.  Finally, $d\ge2$ implies
\[
K-d-1\le K-3<K-2,
\]
so the full population is lower order than the shadow.  This proves
\eqref{eq:surplus-asymptotic}.
\end{proof}

For $K\ge3$ and a target pair $T=\{p,q\}$ with $p<q$, the dimension gap takes
the explicit form
\begin{equation}
\frac{R_{K,\{p,q\}}(H;B)}{\Full_{K,\{p,q\}}(pH)}
\sim
\frac{p+q-2}{p(K-2)}\log\log H.
\label{eq:pair-ratio}
\end{equation}
The factors $1/(p-1)$ and $1/(q-1)$ come from the exact geometric sums over
prime-power exponents.

\section{Completion of the proof}
\label{sec:finish}

\begin{proof}[Proof of Theorem~\ref{thm:main-intro}]
Suppose, toward a contradiction, that $\operatorname{Sat}(T)$ fails for some
allowed exact support $T$.  Section~\ref{sec:defective} shows that the case
$|T|=K$ is impossible.  Hence
\[
2\le d=|T|<K.
\]

By Lemma~\ref{lem:near-height}, selected proper values are unbounded.  There
are therefore arbitrarily large proper-value records $H$ occurring after the
fixed full-return cutoff $n_0$.  At each such record,
\eqref{eq:record-necessary} gives
\begin{equation}
R_{K,T}(H;B)-\Full_{K,T}(pH)\le C_T,
\label{eq:final-upper}
\end{equation}
where $B$ is the actual two-back support and $C_T$ is the fixed number of
proper terms in the opening segment.

On the other hand, $B\cap T=\varnothing$ and $|B|\le K$, so
Theorem~\ref{thm:record-surplus} applies uniformly to the actual blocker at
every record and yields
\[
R_{K,T}(H;B)-\Full_{K,T}(pH)\longrightarrow+\infty
\]
as $H\to\infty$ through the proper records.  This contradicts
\eqref{eq:final-upper}.  Therefore $\operatorname{Sat}(T)$ holds for every
allowed support $T$.

Every allowed exact-support queue is consequently exhausted.  Hence every
integer with $2\le\omega(m)\le K$ occurs, completing the proof.
\end{proof}

\subsection{Why proper records are the effective observation points}

Two features of the $T$-tail make proper records the natural observation
points.  First, the episode grammar associates every proper tail term with the
distinct full starter immediately preceding it, giving
$N_{\mathrm{prop}}\le N_{\mathrm{full}}$.  Second,
Lemma~\ref{lem:near-height} associates every full tail term $v$ with an earlier
proper value $w(v)>v/p$.  Thus, if $H(N)$ is the running maximum of the proper
values selected by index $N$, every full term selected by then is below
$pH(N)$.

At an arbitrary proper term $H=a_n$, the running maximum $H(n)$ may be larger
than $H$, so these two counts need not be controlled at the same numerical
scale.  At a proper-value record, however, $H(n)=H$.  The local picture is then
free--full--proper, and the exact-rank-$K$ shadow candidates have novelty
automatically.  Record times therefore synchronize the historical proper count
with the full-value envelope at the common parameter $H$, producing the fixed
upper bound \eqref{eq:record-necessary}.  The analytic argument of
Section~\ref{sec:counting} shows that this synchronized upper bound cannot
persist for arbitrarily large records.

\subsection{Dependence on the finite rank cap}

The fixed cap enters in two places.  First, maximal-rank shadow candidates
cannot have their support contained in a full predecessor, giving automatic
novelty in Lemma~\ref{lem:automatic-novelty}.  Second, the two-back support has
at most $K$ primes, so Lemma~\ref{lem:moving-mask} makes its total deletion cost
uniformly lower order.  Neither statement remains uniform when the rank cap is
removed; Theorem~\ref{thm:main-intro} therefore makes no claim about the
unrestricted process.

\section*{AI disclosure}
AI tools, specifically ChatGPT, were used throughout all stages of this
research, including computational experiments, theoretical synthesis, and
multiple rounds of AI-assisted proofreading.

\clearpage
\bibliographystyle{jis}
\bibliography{references}
\end{document}